\documentclass[reqno]{amsart}
\usepackage{amsfonts,amsthm,amsmath,amssymb,amscd,amsmath,epsf,latexsym,enumerate,mathrsfs}
\usepackage{graphicx}
\usepackage[utf8x]{inputenc}
\newtheorem{theorem}{Theorem}
\newtheorem{proposition}[theorem]{Proposition}
\newtheorem{lemma}[theorem]{Lemma}
\newtheorem{definition}{Definition}
\newtheorem{corollary}{Corollary}
\newtheorem{example}{Example}
\newtheorem{remark}{Remark}

\newtheorem{conjecture}{Conjecture}

\newcommand{\bC}{\mathbb{C}}
\newcommand{\ee}{\end{equation}}

\newcommand {\cC} {\mathcal C}

\newtheorem{CONJ}{Conjecture}

\begin{document}
          \numberwithin{equation}{section}

          \title[Around a conjecture of K.~Tran]
          {Around a conjecture of K.~Tran}

          \author[I.~Ndikubwayo]{Innocent Ndikubwayo}
\address{Department of Mathematics, Stockholm University, SE-106 91
Stockholm,         Sweden}
\email {innocent@math.su.se}

\address{Department of Mathematics, Makerere University, 7062 Kampala, Uganda}

\email {ndikubwayo@cns.mak.ac.ug}


\begin{abstract} 

We study the root distribution of a sequence of polynomials $\{P_n(z)\}_{n=0}^{\infty}$  with the rational generating function $$ \sum_{n=0}^{\infty} P_n(z)t^n= \frac{1}{1+ B(z)t^\ell +A(z)t^k}$$ for $(k,\ell)=(3,2)$ and $(4,3)$  where $A(z)$ and $B(z)$ are arbitrary polynomials in $z$ with complex coefficients. We show that the zeros of $P_n(z)$ which satisfy $A(z)B(z)\neq 0$ lie on a   real algebraic curve which we describe explicitly.


\end{abstract}

\maketitle

\section{Introduction}
The study of zeros of sequences of polynomials plays an important role in many areas of mathematics such as analysis, probability theory, combinatorics and geometry. In this article, we study the distribution of zeros of polynomials in a polynomial sequence  $\{P_n\}_{n=0}^{\infty}$ generated by a certain type of three-term recurrence relation. Such polynomials are of  interest due to several remarkable properties they possess. For example, three-term recurrence relations are useful in numerical mathematics for producing sequences of orthogonal polynomials.
\medskip
Let $Q_1(z),Q_2(z) \ldots, Q_k(z)$ be fixed complex polynomials and $\{P_n(z)\}_{n=0}^{\infty}$ be the sequence of polynomials generated by recurrence relation of the form:  
\begin{equation} \label{eqn1}
P_n(z)+ Q_1(z)P_{n-1}(z)+ Q_2(z)P_{n-2}(z)+ \dots + Q_k(z)P_{n-k}(z)=0,~~~~~~~n=1,2,\dots
\end{equation}
subject to certain initial conditions. 

The problem of describing the location of the zeros of polynomials in the sequence $\{P_n(z)\}_{n=0}^{\infty}$ might have two different versions. The first is asymptotic and it aims at finding the limiting curve for the zeros of $P_n(z)$ as $ n \to  \infty$. This is the approach taken in the papers \cite{KK, KK2, BG, Boy}. The second is of exact type; it aims at finding the curve where all the zeros of $P_n(z)$ lie for all $n$ (or at least for all large $n$), for example \cite{T, TI}. For general recurrence relations, such curves do not exist. However, for three-term recurrences with $2 \leqslant k \leqslant 5$ and appropriate initial conditions, such a curve containing all the zeros of $P_n(z)$ exists and is given in  \cite{T}. More generally in \cite[Conjecture 6]{T}, K. Tran conjectured existence of such a curve for all $k \geqslant 6$. Below we reformulate Tran's Conjecture.

\begin{CONJ}\label{conj:Tran} For an arbitrary pair of complex polynomials $A(z)$ and $B(z)$, every zero of every polynomial in the sequence $\{P_n(z)\}_{n=1}^\infty$ satisfying the three-term recurrence relation of length $k$
$$P_n(z)+B(z)P_{n-1}(z)+A(z)P_{n-k}=0 \quad \mbox{for} \quad n=1,2,\dots
$$
with the standard initial conditions $P_0(z)=1$, $P_{-1}(z)=\dots=P_{1-k}(z)=0$ which is  not a zero of  $A(z)$ lies on the portion of the real algebraic  curve  $\cC\subset \bC$ given by 
$$\Im \left(\frac{B^k(z)}{A(z)}\right)=0\quad {\rm and}\quad 0\le (-1)^k\Re \left(\frac{B^k(z)}{A(z)}\right)\le \frac{k^k}{(k-1)^{k-1}}.$$
\end{CONJ}

The first part of the  above conjecture explicitly defines the real algebraic curve on which all the zeros of the polynomials $P_n(z)$ (except the zeros shared with $A(z)$) are located. The second part describes the exact portion on this curve where these zeros lie. In \cite{T}, Conjecture~\ref{conj:Tran} was settled for $k\le 5$. In a subsequent paper \cite{TI}, K. Tran  proved existence of such a curve containing all the zeros of $P_n(z)$ for sufficiently large $n$ and arbitrary $k$. More supporting results can be found in \cite{Fr}. 
 Additionally, a criterion for the reality of all the zeros of every polynomial in the sequence $\{P_n(z)\}_{n=0}^\infty$ for $k=2$ can be found in \cite{Nd}. 

\medskip
Based on  numerical experiments,  B.~Shapiro has formulated the following generalization of  Conjecture~\ref{conj:Tran} (private communication).

\begin{conjecture} \label{conj:Shapiro} For an arbitrary pair of complex polynomials $A(z)$ and $B(z)$, every zero of every polynomial in the sequence $\{P_n(z)\}_{n=0}^\infty$ satisfying the three-term recurrence relation of length $k$
\begin{align}\label{asb}
P_n(z)+B(z)P_{n-\ell}(z)+A(z)P_{n-k}(z)=0 \quad \mbox{for} \quad n=1,2,\ldots
\end{align}
with the standard initial conditions $P_0(z)=1$, $P_{-1}(z)=\dots=P_{1-k}(z)=0$ and coprime $1\le \ell <k$ which is  not a zero of $A(z)$ or $B(z)$ 
lies on the real algebraic  curve  $\cC\subset \bC$ given by 
$$\Im \left(\frac{B^k(z)}{A^\ell(z)}\right)=0.$$
\end{conjecture}

The purpose of this paper is to prove  some specific cases of  Conjecture $1$ and make them more precise.
It is important to note that the rational function $B^k (z)/A^\ell(z)$ can be written in terms of the discriminant of the denominator of the generating function of the polynomial sequence generated by \eqref{asb}. Our approach to the proof of the specific cases of Conjecture $1$  uses the \textit{q}-analogue of the discriminant of a polynomial, a concept introduced by Ismail in \cite{Is} and used in \cite{T}.  
 In particular, the ratios of zeros of the denominator of the generating function with equal modulus will play a fundamental role. 
 
 \medskip
The paper is organized as follows. In Section \ref{B}, we review the most important notions and basic results. In Section \ref{BI}, we prove some relevant general results. In Section \ref{BII}, we prove the main result of the paper. In Section \ref{BIII}, we give some numerical examples.

\section{Discriminants}\label{B}
This paper uses discriminants, and therefore we present below a few key results about them from the literature. Our initial step is to recall the concept of discriminants of polynomials, both ordinary and \textit{q}-discriminants.

 
\begin{definition}[see \cite{Ge}]
Let $P(x)$ be a univariate polynomial of degree $n$ with zeros $x_1, \dots, x_n$ and leading coefficient $a_n$. The ordinary discriminant of $P(x)$  is  defined as  $$Disc_x(P(x))=a_n^{2n-2} \prod_{1\leq i<j\leq n}(x_i-x_j)^2.$$  
\end{definition} 
This ordinary discriminant of $P(x)$ can also be expressed in terms of the resultant of $P(x)$ and its derivative.  Let $P(x)$, $Q(x)$ be polynomials of degrees $n$, $m$ and leading coefficients $a_n$, $b_m$ respectively. The resultant of $P(x)$ and $Q(x)$ is defined by 
\begin{equation} \label{res}
Res(P(x),Q(x)) = a_n^m \prod_{P(x_i)=0} Q(x_i) = b_m^n \prod_{Q(x_i)=0} P(x_i). 
\end{equation}
The discriminant of $P(x)$ is then computed as follows
\begin{equation}\label{res1}
 Disc_x(P(x)) =(-1)^{\frac{n(n-1)}{2}}\frac{1}{a_n}  Res(P(x),P'(x)).
 \end{equation}
 For more details  on the ordinary discriminants and resultants, see \cite{KD, Ge}.
\medskip
\begin{example} \label{dis} 
If $f(x)$ and $g(x)$ are fixed complex polynomials in $x$ with $f\not\equiv 0$, then the discriminant of $P(x,t)=f(x)t^3+ g(x)t^2 +1$ as a polynomial in $t$ is given by \begin{eqnarray*} \label{eqq}
Disc_t(P(x,t))=-27(f(x))^2-4(g(x))^3.
\end{eqnarray*}
\end{example}
The proof/ verification of the above example follows directly from equation \eqref{res1}.

\medskip
Generally, the discriminant of a polynomial connects with the ratio of its zeros in the sense that the discriminant  is zero if and only if  the polynomial has a multiple zero. In particular, the discriminant of a polynomial vanishes  whenever there exist at least two zeros that are equal, i.e., zeros whose  ratio is $1$. 

\begin{definition}[see \cite{Is}]\label{qdisc} 
Let $P(t)$ be a polynomial of degree $n$ with zeros $x_1, \ldots, x_n$ and leading coefficient $a_n$. The \textit{q}-discriminant of  $P(t)$ is
\begin{eqnarray*}
Disc_t(P(t);q)=q^{n(n-1)/2}a_n^{2n-2} \prod_{1\leq i<j\leq n}(q^{-1/2}x_i-q^{1/2}x_j)(q^{1/2}x_i-q^{-1/2}x_j).\end{eqnarray*}
\end{definition}
In other words, 
\begin{eqnarray}  \label{dff}  Disc_t(P(t);q)=q^{n(n-1)/2}a_n^{2n-2} \prod_{1\leq i<j\leq n}\left(
x_i^2+ x_j^2-(q^{-1}+q)x_ix_j\right).
\end{eqnarray}
It is clear from Definition \ref{qdisc} that $Disc_t(P(t);1)=Disc_t(P(t))$, the ordinary discriminant of the polynomial. Additionally, the \textit{q}-discriminant vanishes whenever there is a pair of zeros with ratio $q$. Below we give some examples of \textit{q}-discriminants.
\begin{example} Let  $P(t)= at^2 + bt + c$ be a quadratic polynomial with complex coefficients. The \textit{q}-discriminant of $P(t)$ is  $$ Disc_t(P(t);q)=b^2q-(q+1)^2ac.$$ 
\end{example}

This case can easily be computed from Equation (\ref{dff}). 
\medskip A similar but tedious calculation give the $q$-discriminant of cubic polynomial $P(t)= at^3 + bt^2 + ct+d$ as
$$ Disc_t(P(t);q)= -a^2 d^2 (1 + q + q^2)^3 - b^2 q^2 (-c^2 q + b d (1 + q)^2) + 
 a c q (-c^2 q (1 + q)^2 + b d \Phi)$$
\textit{ where} $\Phi= 1 + 5 q + 6 q^2 + 5 q^3 + q^4.$

\medskip
In general computing \textit{q}-discriminants using Definition \ref{qdisc} can be a tedious exercise. However, the following proposition proved in \cite{Is} is often used.
\begin{proposition} \label{call} 
Let $P(t)$ be a polynomial of degree $n$ with zeros $x_1,\ldots,x_n$ and leading coefficient $a_n$. The \textit{q}-discriminant of $P(t)$ is given by:
$$Disc_t(P(t); q) = (-1)^{n(n-1)/2} a_n^{n-2}\prod_{i=1}^n(D_qP)(x_i),$$
where
$$(D_qP)(t) = \frac{P(t) - P(qt)}{t-qt}.$$
\end{proposition}
Proposition \ref{call} is used in \cite{TI, RK} to derive the expression of the \textit{q}-discriminant of the polynomial $D(t,x) = A(x)t^n + B(x)t + 1$. This is stated in Proposition \ref{cal}, (we suppress the parameter $x$).

\begin{proposition} [K. Tran \cite{TI}] \label{cal} 
Let $D(t)= A t^n + Bt +1$ be a polynomial  in $t$ of degree $n$ where $A$ and $B$ are arbitrary fixed complex functions. The \textit{q}-discriminant of $D(t)= At^n + Bt +1$ is  given by 
$$Disc_t(D(t); q)= \pm A^{n-2}\left( B^n\frac{ q^{n-1}(1-q^{n-1})^{n-1}}{(1-q)^{n-1}}   + (-1)^{n-1} \frac{ (1-q^n)^n}{(1-q)^n}A \right).$$
\end{proposition}

\medskip For completeness let us review some definitions (also obtained in \cite{Sok, Boy}) about
the root distribution of a sequence of functions
$$f_m=\sum_{i=1}^{n}\alpha_i(z)\beta_i(z)^m,$$
where $\alpha_i(z)$ and $\beta_i(z)$ are analytic in a domain $D$. Let us call an index $i$ dominant at $z$ if
$|\beta_i(z)|\geq | \beta_j(z)|$ for all $j (1 \leq j \leq n)$. Let
$$D_i = \{ z \in D : i~~\mbox{is~~dominant~~at}~~z\}.
$$
Denote $Z(f_m)$ the  set of zeros of $f_m(z)$. 
\begin{definition}\label{yui}
The set of all $z\in D$ such that every neighborhood $U$ of $z$ has a non-empty intersection with infinitely many of the sets $Z(f_m)$ is called limit superior of $Z(f_m)$ and is denoted by
$\limsup Z(f_m)$.  On the other hand, the set of all $z\in D$ such that every neighborhood $U$ of $z$ has a non-empty intersection with all but finitely many of the sets $Z(f_m)$ is called limit inferior of $Z(f_m)$ and is denoted by $\liminf Z(f_m)$.
\end{definition}

We note that both $\liminf Z(f_m)$ and $\limsup Z(f_m)$ always exist. In addition, it always holds that
 $\liminf Z(f_m)\subset  \limsup Z(f_m)$. Moreover, in the event that both the limit superior and the limit inferior of $Z(f_m)$ coincide, then we call this set, the limit of the $Z(f_m)$, namely,
$$\lim Z(f_m)=\liminf Z(f_m)= \limsup Z(f_m).$$
%
We now state the following theorem, see ([9, Theorem 1.5]) which will be useful in this work.
\begin{theorem}
\label{xxx}
 Let $D$ be a domain in $\mathbb{C}$, and let $\alpha_1, \dots, \alpha_n,  \beta_1, \dots, \beta_n, (n \geq 2)$ be analytic functions on $D$, none of which is identically zero. Let us further assume a ``no-degenerate-dominance" condition, that is, there do not exist indices $i \neq  i'$ such that $\beta_i \equiv \omega \beta_{i'}$ for some constant $\omega$ with $|\omega| = 1$ and such that $D_i (= D_{i'})$ has nonempty interior. For each integer $m \geq 0$, define $f_m$ by $$f_m=\sum_{i=1}^{n}\alpha_i(z)\beta_i(z)^m.$$
Then $\liminf Z(f_m) = \limsup Z(f_m)$, and a point $z$ lies in this set if and only if either

\begin{enumerate}[(a)]
\item there is a unique dominant index $i$ at $z$, and $\alpha_i(z) = 0,$ or
\item there are two or more dominant indices at $z$.
\end{enumerate}
\end{theorem}

\begin{remark}\label{label}
\end{remark}
If $z^* \in \mathbb{C}$ is a fixed complex number such that the zeros in $t$ of $D(t,z^*)=A(z^*)t^k+B(z^*)t^\ell +1$ are distinct, then by use of partial fraction decomposition and Theorem \ref{xxx}, $z^*$ belongs to $\liminf Z(f_m)$ when the two smallest zeros (in modulus) of $D(t,z^*)$ have the same modulus. Note that the functions $t_0(z), t_1(z), \ldots, t_k(z)$ are analytic in a neighborhood of $z^*$ by the implicit function theorem \cite{zzz}.

Now with our setup, if we can obtain a point $z^*\in \mathbb{C}$ so that the zeros of $D(t,z^*)$ are distinct and the two smallest (in modulus) zeros of $D(t,z^*)$ have the same modulus, then for such a point $z^*$, we have that $z^* \in \liminf Z(f_m)= \limsup Z(f_m)$. This implies that on a small neighborhood of $z^*$,  there is a zero of $f_m$ for all large $m$ by the Definition \ref{yui} of $\liminf Z(f_m)$. For details see \cite{zzz}.

It is important to note that this elegant theorem provides a description of the asymptotic behaviour of the zeros of $\{P_n(x)\}$ in the general case. On the other hand, Conjectures \ref{conj:Tran} and  $1$ describe specific cases  where all the  zeros of  $P_n(x)$ (different from those of $A(x)$) actually lie on the 
 curve $\cC$ for all $n$ or for all sufficiently large $n$.

\section{GENERAL RESULTS} \label{BI}

In order to prove some  specific cases of Conjecture $1$, we first prove some general auxiliary lemmas. The results for cases when $(k,\ell)=(3,2)$ and $(k,\ell)=(4,3)$ then fit in as specific cases.  In this section among other things, we compute the \textit{q}-discriminant of a special polynomial  $D(t, z)= A(z)t^k+ B(z)t^\ell+1$ for coprime $1\le \ell <k$. We begin as follows.

\begin{lemma} \label{nyi} 

Let $A(z), B(z)$ be fixed complex polynomials and $\{P_n(z)\}$ be a sequence of polynomials given by linear 
recurrence relation of length $k$
\begin{equation}\label{uutii}
P_n(z)+B(z)P_{n-\ell}(z)+A(z)P_{n-k}(z)=0
\end{equation}
with the standard initial conditions $P_0(z)=1$, $P_{-1}(z)=\dots=P_{1-k}(z)=0$ where $1\le \ell <k$ are  coprime.
The generating function of $\{P_n(z)\}_{i=0}^{\infty}$  is given by

\begin{equation}\label{uut}
 \sum_{n=0}^{\infty}P_n(z)t^n= \frac{1}{1+B(z)t^\ell+ A(z)t^k}.
 \end{equation}
\end{lemma}
\begin{proof}
Let $G(t)=\sum_{n\geq 0}P_nt^n$ be the generating function of $\{P_n\}_{i=0}^{\infty}$. Multiplying \eqref{uutii} by $t^n$ and summing over from $n=k$ to infinity, we obtain
\begin{align*}
\sum_{n\geq k}P_nt^n+ B\sum_{n\geq k}P_{n-\ell}t^n+A \sum_{n\geq k}P_{n-k}t^n=0.
\end{align*}
Simplifying gives
\begin{align*}
\left(G(t)- \sum_{n=0}^{k-1}P_nt^n\right)+ B t^\ell\left(G(t)- \sum_{n=0}^{k-\ell -1}P_nt^n\right)+ At^kG(t)=0
\end{align*}
which is equivalent to
\begin{align*}
G(t)&=\frac{1}{1+B(z)t^\ell+ A(z)t^k}\left(\sum_{n=0}^{k-1}P_nt^n+ Bt^\ell \sum_{n=0}^{k-\ell -1}P_nt^n \right)\\
&=\frac{1}{1+B(z)t^\ell+ A(z)t^k}\left(1+ \sum_{n=1}^{\ell-1}P_n t^n   + \sum_{n=\ell}^{k-1}(P_n + BP_{n-\ell})t^n\right).
\end{align*}
From Equation (\ref{uutii}), we have $P_n=0$ for $0< n \leq \ell-1$ and $P_n + BP_{n-\ell}=0$ for $\ell \leq n \leq k-1$.  It follows that $1+ \sum_{n=1}^{\ell-1}P_n t^n   + \sum_{n=\ell}^{k-1}(P_n + BP_{n-\ell})t^n=1$, which proves the required result.
\end{proof}

Let us now consider the \textit{q}-discriminant of $D(t,z)$. In the following proposition, we shall suppress the variable $z$.  
\begin{theorem} \label{ipt1} For coprime $1\le \ell <k$, the $q$-discriminant of  $D(t)= At^k+ Bt^\ell+1$  is given by
\begin{align} \label{In4ii}
Disc_t(D(t);q)=  (-1)^{\frac{k(k+1)}{2}}\left((q^k-1)^{k}A^\ell -B^k (1-q^\ell)^{\ell}(q^\ell -q^k)^{k-\ell}\right) W
\end{align}
where $W=A^{k-\ell -1} B^{\ell-1} (1-q)^{-k}$.
\end{theorem}

\begin{proof}
By Proposition \ref{call}, we have \begin{eqnarray*}
Disc_t(D(t); q)=  (-1)^{k(k-1)/2} A^{k-2} \prod_{\substack{D(s_i^*)=0\\ 1\leq i \leq k }}(D_qD)(s_i^*) ,\end{eqnarray*}
where $$(D_qD)(t)= \frac{D(t)-D(qt)}{t-qt}.$$
Now  \begin{eqnarray*} 
(D_qD)(t)&=& \frac{A t^k+ B t^\ell +1-(Aq^k t^k+ Bq^\ell t^\ell +1)}{t-qt}
\end{eqnarray*}
\begin{eqnarray}  \label{wo}
 &=&A t^{k-1}\frac{1-q^k}{1-q} + B t^{\ell-1}\frac{1-q^\ell}{1-q}.\end{eqnarray}
 Using equation (\ref{res}) with $P=D$ and $Q=D_qD$ we obtain 
$$\prod_{\substack{D(s_i^*)=0\\ 1\leq i \leq k }}(D_qD)(s_i^*)=A\left( \frac{1-q^k}{1-q}\right)^k\prod_{\substack{(D_qD)(t_i^*)\\ 1\leq i \leq k-1 }}D(t_i^*).$$
This implies that 
\begin{align}
Disc_t(D(t);q)&= (-1)^{\frac{k(k-1)}{2}} A^{k-1} \left(\frac{1-q^k}{1-q}\right)^k \prod_{\substack{(D_qD)(t_i^*)\\ 1\leq i \leq k-1 }}D(t_i^*) \nonumber\\ &=(-1)^{\frac{k(k-1)}{2}} A^{k-1} \left(\frac{1-q^k}{1-q}\right) \prod_{\substack{(D_qD)(t_i^*)\\ 1\leq i \leq k-1 }}\frac{1-q^k}{1-q}D(t_i^*) \label{lab}.
\end{align}
Now, $D(t)$ and $D_qD(t)$  are related by  the equation
\begin{eqnarray*}
(D_qD)(t)&= \left(\frac{D(t)}{t}-Bt^{\ell -1} -\frac{1}{t}\right)  \frac{1-q^k}{1-q}+ Bt^{\ell -1}\left(  \frac{1-q^{\ell}}{1-q}\right).
\end{eqnarray*}
Therefore
\begin{align*}
D(t)\left( \frac{1-q^k}{1-q}\right)&=(B t^{\ell}+1)\left(  \frac{1-q^k}{1-q}\right)- B t^{\ell}\left(\frac{1-q^{\ell}}{1-q}\right) + t (D_qD)(t) 
\\&= t (D_qD)(t) + B t^{\ell}\left(  \frac{q^\ell -q^k}{1-q}
\right) + \frac{1-q^{k}}{1-q}.
\end{align*}
At the zeros $t_i^*$ of $(D_qD)(t)$ for $1\leq i \leq k-1$, we have $(D_qD)(t_i^*)=0$, hence 
\begin{align}\label{xy}
D(t_i^*)\left( \frac{1-q^k}{1-q}\right)= (t_i^*)^{\ell}\left(  \frac{q^\ell -q^k}{1-q} \right)B + \frac{1-q^{k}}{1-q}.
\end{align}
Now substituting equation \eqref{xy} into equation \eqref{lab} we obtain 
\begin{align*}
Disc_t(D(t);q)&= (-1)^{\frac{k(k-1)}{2}} A^{k-1} \left( \frac{1-q^k}{1-q}\right) \prod_{\substack{(D_qD)(t_i^*)\\ 1\leq i \leq k-1 }} \left((t_i^*)^{\ell}\left(  \frac{q^\ell-q^k}{1-q} \right)B + \frac{1-q^{k}}{1-q}\right) \\ &= (-1)^{\frac{k(k-1)}{2}} (AB)^{k-1}\left(\frac{1-q^k}{1-q}\right)\left( \frac{q^\ell -q^k}{1-q} \right)^{k-1}\\ &\qquad \times \prod_{\substack{(D_qD)(t_i^*)\\ 1\leq i \leq k-1 }}
\left((t_i^*)^\ell + \frac{1-q^k}{B(q^\ell -q^k)}\right).
\end{align*}
From Equation (\ref{wo}), the condition $(D_qD)(t_i^*)=0$  for $1\leq i \leq k-1$ implies that $0$ is a root that occurs with multiplicity $\ell-1$. The other zeros satisfy the relation
\begin{eqnarray} \label{In2}
(t_i^*)^{k-\ell}=  -\frac{B}{A}\left( \frac{1-q^\ell}{1-q^k} \right).
\end{eqnarray}
Since the integers $k$ and $\ell$ are coprime which is equivalent $k-\ell$ and $\ell$ being coprime, we can make the change of variables $\xi= t^{\ell}$. We observe that $t_i^*$ for $1\leq i \leq k-\ell$ is a zero to the polynomial $r(t)=t^{k-\ell}+\frac{B}{A}\left( \frac{1-q^\ell}{1-q^k} \right)$. By the change of variable $\xi= t^{\ell}$, the polynomial $r(t)$ becomes
\begin{align} \label{In3}
\Psi(\xi)=\xi ^{k-\ell}- \left(\frac{-B}{A} \right)^\ell \left( \frac{1-q^\ell}{1-q^k}\right)^\ell.
\end{align}
In particular, the solutions $\xi_i^*$ to $\Psi(\xi)=0$ for $1\leq i \leq k-\ell$ are of the form $(t_i^*)^{\ell}$ where $t_i^{\ast}$ is a solution to $r(t)=0$, (nonzero solution to $(D_qD)(t)=0$). So 
\begin{align*}
\prod_{\substack{(D_qD)(t_i^*)=0\\ t_i^*\neq 0 \\1\leq i \leq k-\ell}} \left((t_i^*)^\ell + \frac{1-q^k}{B(q^\ell -q^k)}\right)=\prod_{\substack{\Psi(\xi_i^*)=0\\1\leq i \leq k-\ell}} \left(\xi_i^* + \frac{1-q^k}{B(q^\ell -q^k)}\right).
\end{align*}
The zero $t_i^{\ast}=0$ for $1\leq i\leq \ell -1$ contributes the following factor to the product of the discriminant
\begin{align*}
\left(\frac{1-q^k}{B(q^\ell -q^k)}\right)^{\ell -1}
\end{align*}
The expression for $Disc_t(D(t);q)$ now has the form

\begin{align*} \label{In4}
Disc_t(D(t);q)&=(-1)^{\frac{k(k-1)}{2}} (AB)^{k-1}\left(\frac{1-q^k}{1-q}\right)\left( \frac{q^\ell -q^k}{1-q} \right)^{k-1}\left(\frac{1-q^k}{B(q^\ell -q^k)}\right)^{\ell -1} \\&\qquad \times\prod_{\substack{\Psi(\xi_i^*)=0\\1\leq i \leq k-\ell}} \left(\xi_i^* + \frac{1-q^k}{B(q^\ell -q^k)}\right).
\end{align*}
We now simplify the remaining product.
Consider the following equation:
\begin{eqnarray} \label{one1y}
\xi^{k-\ell}=  \left(\frac{-B}{A} \right)^\ell \left( \frac{1-q^\ell}{1-q^k}\right)^\ell.
\end{eqnarray}
Using the change of variable  \begin{eqnarray*} \label{one21}
 \xi=  Y - \frac{1-q^k}{B(q^\ell -q^k)},
\end{eqnarray*}
Equation (\ref{one1y}) becomes
\begin{eqnarray*} \label{one23}
\left( Y - \frac{1-q^k}{B(q^\ell -q^k)}\right)^{k-\ell}-\left(\frac{-B}{A} \right)^\ell \left( \frac{1-q^\ell}{1-q^k}\right)^\ell=0.\end{eqnarray*}
Using Binomial Theorem, we get
\begin{equation*} \label{eh12}
g(Y):= Y ^{k-\ell}+ \dots + \left(- \frac{(1-q^k)}{B(q^\ell-q^k)} \right)^{k-\ell} - \left(\frac{-B}{A} \right)^\ell \left(  \frac{1-q^\ell}{1-q^k}\right)^\ell=0.
\end{equation*} 
The constant term in $g(Y)$ is 
\begin{equation*}
C= \left(- \frac{(1-q^k)}{B(q^\ell-q^k)} \right)^{k-\ell} - \left(\frac{-B}{A} \right)^\ell \left( \frac{1-q^\ell}{1-q^k}\right)^\ell.
\end{equation*}
For $i=1,2,\dots, k-\ell$, let $y_i$ are the solutions to $g(Y)=0$. Then
\begin{eqnarray*}
\xi_j=  y_i - \frac{1-q^k}{B(q^\ell -q^k)}
\end{eqnarray*} 
for some $j\in \{1,2,\dots, k-\ell \}$. So we can choose $i=j$ so that $\xi_i=  y_i - \frac{1-q^k}{B(q^\ell -q^k)}$.

Now the product of solutions to $g(Y)=0$ is given by
\begin{eqnarray*} 
C=(-1)^{k-\ell}\prod_{i=1}^{k-\ell} y_i= \left[   \left(- \frac{(1-q^k)}{B(q^\ell-q^k)} \right)^{k-\ell} - \left(\frac{-B}{A} \right)^\ell \left( \frac{1-q^\ell}{1-q^k}\right)^\ell\right].
\end{eqnarray*}
or
\begin{align}\label{kataala}
\prod_{j=1}^{k-\ell} \left(\xi_j + \frac{1-q^k}{B(q^\ell -q^k)}\right) =\prod_{i=1}^{k-\ell} y_i= (-1)^{k-\ell}C.
\end{align}
Plugging equation (\ref{kataala})   into the expression for $Disc_t(D(t);q)$ gives
\begin{align} \label{Inkabu}
Disc_t(D(t);q) &= (-1)^{\frac{k(k-1)}{2}} \frac{(AB)^{k-1}(q^l-q^k)^{k-1}(1-q^k)}{(1-q)^k} \left(\frac{1-q^k}{B(q^\ell -q^k)}\right)^{\ell -1}\nonumber \\&\qquad \times (-1)^{k-\ell}\left[ 
\frac{(q^k-1)^{k-\ell}}{B^{k-\ell}(q^l-q^k)^{k-\ell}} -\frac{(-1)^{\ell}B^{\ell}(1-q^{\ell})^{\ell}}{A^{\ell}(1-q^k)^{\ell}}
\right].
\end{align}
Simplifying equation (\ref{Inkabu})  gives 
\begin{eqnarray*} 
\label{In4i}Disc_t(D(t);q)=  (-1)^{\frac{k(k+1)}{2}}\left((q^k-1)^{k}A^\ell- B^k (1-q^\ell)^{\ell}(q^\ell -q^k)^{k-\ell}\right)W,
\end{eqnarray*}
where $W= A^{k-\ell -1} B^{\ell-1} (1-q)^{-k}$. The proof is complete.

\end{proof}

\begin{corollary} \label{corr} Let $A$ and $B$ be fixed complex polynomials.
The $q$-discriminant of  $D(t)= At^3+ Bt^2+1$ is 
\begin{align*} 
Disc_t(D(t);q)= \frac{B \left(A^2 \left(q^3-1\right)^3+B^3 (q-1) q^2 \left(q^2-1\right)^2\right)}{(1-q)^3}.
 \end{align*}
\end{corollary}
\begin{proof}
Set $k=3$ and $\ell =2$ in the Theorem \ref{ipt1} above and obtain the result.
\end{proof}

The following lemma will be used in the proof of the main results.

\begin{lemma}\label{go2} Let $k,\ell$ be integers and $q \in \mathbb{C}$ with $|q|=1$. The function defined by  $h(q)=\frac{(1-q^k)^k }{(1-q^\ell)^\ell(q^\ell-q^k)^{k-\ell}}$ is real-valued.
\end{lemma}
 \begin{proof}
Since $|q|=1$, we have $q^s-q^{-s}=2 i {\Im}(q^s)$, for any integer $s$, (here $i^2=-1$). It therefore follows that $(1-q^s)^s= (q^{\frac{s}{2}}(q^{-\frac{s}{2}}-q^{\frac{s}{2}}))^s= q^{s(\frac{s}{2})}(-2i)^s(\Im(q^{\frac{s}{2}}))^s.$ So, 
\begin{align*}
h(q)&=\frac{(1-q^k)^k }{(1-q^\ell)^\ell(q^\ell-q^k)^{k-\ell}} \\&= \frac{q^{k(\frac{k}{2})}(-2i)^k(\Im(q^{\frac{k}{2}}))^k}{q^{\ell(\frac{\ell}{2})}(-2i)^\ell(\Im(q^{\frac{\ell}{2}}))^\ell q^{\ell(k-\ell)}q^{(k-\ell)\frac{(k-\ell)}{2}}(-2i)^{k-\ell}(\Im(q^{\frac{k-\ell}{2}}))^{k-\ell}} \\&=
\frac{q^{k(\frac{k}{2})- (\ell(\frac{\ell}{2})+ \ell(k-\ell)+(k-\ell)\frac{(k-\ell)}{2})}(\Im(q^{\frac{k}{2}})^k}{(\Im(q^{\frac{\ell}{2}}))^\ell\cdot (\Im(q^{\frac{k-\ell}{2}}))^{k-\ell}} \\&=
\frac{(\Im(q^{\frac{k}{2}}))^k}{(\Im(q^{\frac{\ell}{2}}))^\ell\cdot (\Im(q^{\frac{k-\ell}{2}}))^{k-\ell}} .
\end{align*} 
Since the expressions $\Im(q^{\frac{k}{2}})$, $\Im(q^{\frac{\ell}{2}})$ and $\Im(q^{\frac{k-\ell}{2}})$ are all real, the result follows.

\end{proof}

\section{Specific cases of conjecture 1 when $(k,\ell)=(3,2)$ and $(k,\ell)=(4,3)$} \label{BII}
In this section, we settle the specific cases of $(k,\ell)=(3,2)$ and $(k,\ell)=(4,3)$ completely by showing that the zeros of $P_n(z)$ lie on a portion of real algebraic curve  $\Im\left( \frac{B^k(z)}{A^\ell(z)}\right)=0.$ We also provide the relevant inequality constraint. 

For the case  $(k,\ell)=(3,2)$, we begin with the following lemma. 
\begin{lemma}\label{kab1}
Suppose $\xi_2, \xi_3 \neq 0$ are complex numbers such that $\xi_2 + \xi_3 + 1 = 0$ and
 \begin{align*} 
  \frac{\xi_2^{n+1}-1 }{\xi_2-1}= \frac{\xi_3^{n+1}-1 }{\xi_3-1}
 \end{align*}
 for some positive integer $n$. Then $\xi_2$ and $\xi_3$  lie on the curve $\Gamma =C_1 \cup C_2 \cup C_3$ and are dense there as  $ n \to \infty$.   Here the  equations of $C_1, C_2$ and $C_3$ (see Fig.1) are given by 
 \[\Gamma =   \left\{\begin{array}{cl} C_1: & x^2+y^2=1 , \hspace*{0.3cm} x \leq -\frac{1}{2}, \\\\
 C_2:&  (x+1)^2+y^2=1 , \hspace*{0.3cm} x \geq -\frac{1}{2}, \\\\  C_3: & x=-\frac{1}{2}+iy, \hspace*{0.3cm} |y| \geq \frac{\sqrt{3}}{2}. \end{array}\right.\] 
\begin{figure}
    \includegraphics[height=7cm, width=7cm]{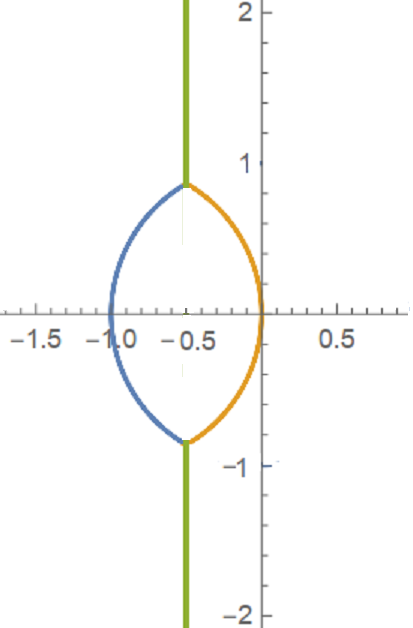}
     \caption{Distribution of the quotients of zeros of $D(t)=At^3+ Bt^2+1$. $C_1$ (blue), $C_2$ (orange) and $C_3$ (green). }\label{see1}
 \end{figure}
\end{lemma}
The proof is similar to the one given in [2, Lemma 2] by letting $\xi \mapsto \frac{1}{\xi}.$
\begin{remark} 
\end{remark}
Note that $t_1, t_2$ and $t_3$ are zeros of $D(t,z)= A(z)t^3+ B(z)t^2+1$. By Vieta's formulae we have $t_1t_2 + t_1t_3+ t_2t_3=0.$ Dividing by $t_1^2$ and letting $q_i=t_i/t_1$ and $q_i=1/\xi_i, ~~~~~~ i= 2, 3$ gives $q_2q_3 + q_2+ q_3=0$ if and only if $\xi_2 + \xi_3 + 1 = 0$, the first condition in Lemma \ref{kab1}. 
\begin{theorem}
Let $ \{P_n(z)\}$ be a polynomial sequence whose generating function is 
\begin{equation*}
 \sum_{n=0}^{\infty}P_n(z)t^n= \frac{1}{1+B(z)t^2+ A(z)t^3},
 \end{equation*}
 where $B(z)$ and $A(z)$ are polynomials in $z$ with complex coefficients. Given any $n \in \mathbb{N}$, those zeros $z$ of $P_n(z)$ for which $A(z)B(z) \neq 0$ lie on the curve defined by 
 \begin{equation*} 
\Im\left( \frac{B^3(z)}{A^2(z)}\right)=0  ~~~~~~~~  \mbox{and}~~~~~~~~~~ 
 0 \leq \Re \left(\frac{B^3(z)}{A^2(z)}\right) < \infty.
  \end{equation*}
\end{theorem}
\begin{proof}
 Let $z_0$ be a zero of $P_n(z)$ such that $A(z_0)B(z_0)\neq 0$ and let $t_1=t_1(z_0),t_2=t_2(z_0), t_3=t_3(z_0)$ be the zeros of $D(t, z_0)= A(z_0)t^3+ B(z_0)t^2+1$. There are two cases we shall consider; namely, the case of repeated zeros and the case of distinct zeros of $D$.

 \textit{Case 1:} Suppose  $D(t,z_0)$ has repeated zeros. The ordinary discriminant of $D(t, z_0)$ is $$Disc_t(D(t,z_0))=-27A^2(z_0)-4B^3(z_0).$$
For repeated zeros,  $Disc_t(D(t,z_0))=0$ which is equivalent to 
 $$\frac{B^3(z_0)}{A^2(z_0)}= -\frac{27}{4} \in \mathbb{R}.$$  Therefore, the point $z_0$ lies on the curve given by $$\Im\left( \frac{B^3(z)}{A^2(z)}\right)=0. $$
 \textit{Case 2:} Suppose all the zeros of $D(t, z_0)$ are distinct, i.e., $Disc_t(D(t, z_0)) \neq 0.$
In this situation, let $t_1, t_2$ and $t_3$ be the three distinct zeros of $D(t,z)$. We first consider the distribution of quotients of zeros $q_i=t_i/t_1, 1\leq i \leq 3$. Then we examine the root distribution of $P_n(z)$ using \textit{q-}discriminants. We  show that these quotients $q_i=t_i/t_1, 1\leq i \leq 3$ lie on the curve in Fig. \ref{see1}.

 Let $z_0$ be a root of $P_{n}(z)$ which satisfies $A(z)B(z)\neq 0$ and let $t_1=t_1(z_0),t_2=t_2(z_0), t_3=t_3(z_0)$ be the zeros of $D(t, z_0)= A(z_0)t^3+ B(z_0)t^2+1$.\\
By partial fraction decomposition,  we have 
$$\frac{1}{1+B(z_0)t^2+ A(z_0)t^3}=\frac{1}{A(z_0)}\left( \frac{1}{(t-t_1)(t-t_2)(t-t_3)}\right)=  \frac{1}{A(z_0)}\sum_{i=1}^{3} \frac{1}{t-t_i}  \prod_{i \neq j} \frac{1}{t_i-t_j}.$$
By using geometric progression, we get 
\begin{align*}
\frac{1}{t-t_i}=-\frac{1}{t_i} \sum_{n=0}^{\infty} \frac{t^n}{t_i^n}= -\sum_{n=0}^{\infty} \frac{1}{t_i^{n+1}}t^n \hspace*{2cm} (|t|<|t_i|). \end{align*}
Hence we obtain
\begin{align}\label{flaq}
\frac{1}{D(t,z_0)}=-\frac{1}{A(z_0)}\sum_{n=0}^{\infty}\left(\sum_{i=1}^{3} \frac{1}{t_i^{n+1}}
\prod_{i \neq j} \frac{1}{t_i-t_j}\right)t^n.
\end{align} 
Comparing coefficients of $t^n$ in equations (\ref{uut}) and (\ref{flaq}) we obtain 
\begin{eqnarray*}\label{yu}
P_n(z_0) = -\frac{1}{A(z_0)}\sum_{i=1}^{3} \frac{1}{t_i^{n+1}}
\prod_{i \neq j} \frac{1}{t_i-t_j}.
\end{eqnarray*} 
Let $q_i=t_i/t_1, ~~~~~~ i= 2, 3$.  Now for any $n \geq 0$, the equation $P_{n}(z_0)=0$ is equivalent to 
 \begin{align}\label{uu}
  q_2^{n}q_3^{n}(q_2-q_3)- q_2^{n}(q_2-1)+ q_3^{n}(q_3-1)=0.\end{align}
Since $q_2-q_3= q_2-1 - q_3+1$  it follows that
  $$q_2^{n}q_3^{n}(q_2-q_3)= q_2^{n}q_3^{n}(q_2-1)- q_2^{n}q_3^{n}(q_3-1).$$ Therefore equation (\ref{uu})  can be rewritten as 
  $$q_2^{n}q_3^{n}(q_2-1)- q_2^{n}q_3^{n}(q_3-1)=q_2^{n}(q_2-1)-q_3^{n}(q_3-1)$$ and consequently since $q_2, q_3 \neq 0, 1$  we obtain 
   \begin{align}\label{uuu} \frac{ q_2^{n}-1}{q_2^n (q_2-1)}= \frac{ q_3^{n}-1}{q_3^n(q_3-1)}.\end{align}
 If we add $1$ on both sides of the equation \eqref{uuu} and set $q_i=1/\xi_i, ~~~~~~ i= 2, 3$ we obtain   [one side]
$$ \frac{q_2^{n}-1}{q_2^n(q_2-1)}+1=\frac{ q_2^{n+1}-1}{q_2^n (q_2-1)}= \frac{q_2^{n+1}(1-(1/q_2)^{n+1})}{q_2^{n+1} (1-1/q_2)}= \frac{\xi_2^{n+1}-1 }{\xi_2-1}.$$
Similarly on the other side of  equation \eqref{uuu}  we obtain  
$$ \frac{ q_3^{n}-1}{q_3^n (q_3-1)}+1= \frac{\xi_3^{n+1}-1 }{\xi_3-1}.$$
Thus  equation (\ref{uuu}) becomes
  \begin{equation*} \label{wan}
  \frac{\xi_2^{n+1}-1 }{\xi_2-1}= \frac{\xi_3^{n+1}-1 }{\xi_3-1},
 \end{equation*}
the second condition in Lemma \ref{kab1}. 
 
 Next, we need show that $\xi_i, i=2,3$ being on the curve $\Gamma$ implies that $q_i \in \Gamma$ as well. In particular, if $\xi_i \in C_1$ then the corresponding $q_i \in C_1$. If $\xi_i \in C_2$ then the corresponding $q_i \in C_3$ and for $\xi_i \in C_3$, the corresponding $q_i \in C_2.$ 
 To do this, it is enough to show that  $\Gamma$ is invariant under the Möbius inversion $z \mapsto \frac{1}{z}$. This is shown as follows.
  
\begin{enumerate}[(i)]\item For any $z=x+iy \in C_1$, we have $|z|=1$ with $\Re(z)\leq -\frac{1}{2}.$ In addition, $z \neq 0$ and so the image of $z$ under inversion $w=\frac{1}{z}$ is $w=\frac{1}{z}$.
Note that $|w|=|\frac{1}{z}|=\frac{1}{|z|}=1$.
Furthermore,  $\Re(w)=\Re( \frac{1}{z})=\frac{x}{x^2+y^2} \leq -\frac{1}{2}.$ Hence  $w \in C_1.$

\item  For any $z=x+iy \in C_2$, we have  $(x+1)^2+y^2=1$ where $x \geq -\frac{1}{2}$.
 Under the  inversion $w= \frac{1}{z}$, we have
 \begin{align*} w= \frac{1}{z}= \frac{\bar{z}}{|z|^2}=\frac{x-iy}{x^2+y^2}=\frac{x-iy}{(x+1)^2+y^2-2x-1}=
 -\frac{1}{2}+ \frac{y}{2x}i. \end{align*} 
Note that on  $C_2$,  $-\frac{1}{2} \leq x \leq 0$ and $-\frac{\sqrt{3}}{2} \leq y \leq \frac{\sqrt{3}}{2}.$ With these ranges on $x$ and $y$, we obtain $|\Im (w)|=|\frac{y}{2x}|\geq \frac{\sqrt{3}}{2}.$ Clearly  $\Re(w)= -\frac{1}{2}$.  Hence $w \in C_3.$

\item For any $z=x+iy \in C_3$, we have $z= -\frac{1}{2}+ y i$ where $|y|\geq \frac{\sqrt{3}}{2}.$ Clearly $z \neq 0$. So $|\bar{z}|=|z|.$
Under the  inversion $w= \frac{1}{z}$, we have
 \begin{align*} w= \frac{1}{z}= \frac{\bar{z}}{|z|^2}=\frac{x}{x^2+y^2}+ i \frac{-y}{x^2+y^2}=\frac{-1/2}{1/4+y^2}+ i \frac{-y}{1/4+y^2}.
  \end{align*} 
%
With the ranges of $x$ and $y$ for $z\in C_3$, we obtain 

\begin{align}  0 \geq \Re(w)= \frac{-1/2}{1/4+y^2}\geq  \frac{-1/2}{1/4+3/4} \geq -\frac{1}{2}.
 \end{align}
In addition, since $\Im (w)= \frac{-y}{1/4+y^2}$ and $|y|\geq \frac{\sqrt{3}}{2}$, we have  $0 \leq |\Im (w)|\leq \frac{\sqrt{3}}{2}$. Therefore  $w \in C_2$.

The conclusion that  $\Gamma$ is invariant under the Möbius inversion $z \mapsto \frac{1}{z}$ thus follows.
 Therefore $q_2$ and $q_3$ lie on the curve $\Gamma$ given in Lemma \ref{kab1}. The proof is complete.
\end{enumerate}   
 Now we know that $q_2$ and $q_3$ are given by the ratio of zeros of $D$ and are therefore the zeros of the discriminant. (Recall from Ismail \cite{Is} that the $q$-discriminant equals zero if there is a pair of zeros with the ratio $q$). From Corollary $\ref{corr}$, we know that the \textit{q}-discriminant of $D(t,z)=1+B(z)t^2 + A(z)t^3$ is given by $$ Disc_t(D(z,q))= \frac{B \left(A^2 \left(q^3-1\right)^3+B^3 (q-1) q^2 \left(q^2-1\right)^2\right)}{(1-q)^3}.$$ Since  $q_2$ and $q_3$ are the zeros of the discriminant then 
$$\frac{B \left(A^2 \left(q_i^3-1\right)^3+B^3 (q_i-1) q_i^2 \left(q_i^2-1\right)^2\right)}{(1-q_i)^3}=0 $$
 with $i=2,3$. Consequently since $q_i \neq 0,1$ we have
$$\frac{B^3(z)}{A^2(z)}= \frac{-(1+q_i+q_i^2)^3}{q_i^2(1+q_i)^2}.$$  
It remains to show that  
 $$f(q)= \frac{B^3(z)}{A^2(z)}= \frac{-(1+q+q^2)^3}{q^2(1+q)^2}$$ maps the curve $\Gamma$ to a real interval so that  we conclude $ \Im(f)=0. $
 
Now since $q$ lies on the curve in Fig. \ref{see1}, we have three possibilities below.

 \begin{enumerate}[(a)]
 \item $|q|=1$, and this corresponds to the points $\xi$ on the curve $C_1$.
 \item $q + \bar{q}= -1$, and this corresponds to the points $\xi$ on the curve $C_2$.
 \item $|-1-q| =1$, and this corresponds to the points $\xi$ on the curve $C_3$.
 \end{enumerate}
 
 For part (a) $$f({q})= -\frac{(1+{q}+{q}^2)^3}{{q}^2(1+{q})^2}=-\frac{q^3(\frac{1}{q}+1+{q})^3}{q^3(\frac{1}{q}+2+{q})} \in \mathbb{R} $$ since $\frac{1}{q}+q \in \mathbb{R}.$ Hence $ \Im(f)=0. $ \\
 
  For part (b), note that $ \bar{q}= -1-q.$ Therefore
 
 \begin{align*}
-\frac{(1+\bar{q}+\bar{q}^2)^3}{\bar{q}^2(1+\bar{q})^2}&= -\frac{(1+ (-1-q)+(-1-q)^2)^3}{(-1-q)^2(1+(-1-q))^2}\\ &=-\frac{(-q+2q+q^2+1)^3}{(-1-q)^2q^2}\\ &=-\frac{(1+q+q^2)^3}{q^2(1+q)^2}
\end{align*}
Consequently $f(\bar{q})=\overline{f(q)}=f(q)$, and we conclude that $ \Im(f)=0. $ \\
 
 For Part (c), $|-1-q| =1\Rightarrow \bar{q}= -\frac{q}{1+q}$.  Therefore 

\begin{align*}f(\bar{q})= -\frac{(1+\bar{q}+\bar{q}^2)^3}{\bar{q}^2(1+\bar{q})^2}= -\frac{(1+ \frac{-q}{1+q}+(\frac{-q}{1+q})^2)^3}{(\frac{-q}{1+q})^2(1+\frac{-q}{1+q})^2}=-\frac{(1+q+q^2)^3}{q^2(1+q)^2}= f(q).\end{align*}
Hence  $f(\bar{q})=\overline{f(q)}=f(q)$, and we conclude that $ \Im(f)=0. $

 The conclusion $ \Im(f)=0$ follows. Therefore, the point $z_0$ lies on the curve given by $$\Im \left( \frac{B^3(z)}{A^2(z)}\right)=0. $$

Next we prove the inequality constraint. From the recurrence relation with $\ell=2, k=3$  we have $$P_n(z)+B(z)P_{n-2}(z)+A(z)P_{n-3}(z)=0$$ whose characteristic equation is $t^3+ B(z)t + A(z)=0.$ The corresponding denominator of the generating function is $D(t,z)=A(z)t^3+ B(z)t^2 + 1.$ Suppressing $z,$ we can write $D(t)=At^3+ Bt^2 + 1.$ 

Let $t_1, t_2, t_3$ be the zeros of $D(t)$ labelled so that $0 < |t_1| \leq |t_2|\leq |t_3|$ and set $q_i:=t_i/t_1$ for  $i=1,2,3$. Note these are functions of $z$. Vieta's formulae give $t_1t_2+ t_1t_3+t_2t_3=0$. This implies that 
$q_3=-\frac{u}{1+u}$  where $u=q_2.$ If $z_0$ is a zero of $P_n(z)$ and $A(z_0)B(z_0)\neq 0$, then by Theorem \ref{xxx}(b) and  Remark \ref{label}, $q_1=1$, $|q_2|=|u|=1$ and $|q_3|\geq 1$.

So, we search $u$ such that the above conditions are satisfied. Let $u=e^{i\theta}$ where $\theta \in [0, 2\pi]$.
It follows from $|q_3|\geq 1$ that $1=|u| \geq |1+u| =2(1+ \cos \theta).$ This implies that $\theta \in \left[\frac{2\pi}{3} , \frac{4\pi}{3} \right].$ 

 Let $\Omega=\left\{ e^{i\theta} :\frac{2\pi}{3} \leq \theta \leq \frac{4\pi}{3}\right\}.$ 
Since $q \in \Omega \subset\Gamma$, it follows that 
\begin{eqnarray*}\label{tod1i}
f(q)=-\frac{(1+q+q^2)^3}{q^2(1+q)^2}.
\end{eqnarray*}

We now compute the range of $f$ on $\Omega$. To do this, we let $q= e^{i \theta}$ where $\theta \in [\frac{2\pi}{3} ,\frac{4\pi}{3}].$
After this parametrization we obtain
\begin{align*}
f(q)=F(\theta)=-\frac{(2 \cos(\theta)+1)^3}{2\cos(\theta)+2}.
\end{align*}
It is clear that  $F$ is well-defined and continuous on the union $[\frac{2\pi}{3},\pi)\cup (\pi, \frac{4\pi}{3}]$. Observe that 
for $\theta \in [\frac{2\pi}{3},\frac{4\pi}{3}]$, we have $-1 \leqslant 2\cos(\theta)+1 \leqslant 0$ and $0\leqslant 2\cos(\theta)+2 \leqslant 1$ hence $F(\theta)\geqslant 0$. Moreover, $F$ attains its minimal values equal to $0$ at $\theta=\frac{2\pi}{3}$ and at $\theta=\frac{4\pi}{3}$. 
Since $\lim_{\theta \to \pi}F(\theta)=+\infty$, it follows from continuity and non-negativity of $F$ that the range of $f$ is $\mathbb{R}_{\geqslant 0}$. 
This proves the inequality condition that $$ 0 \leq \Re \left(\frac{B^3(z)}{A^2(z)}\right) < \infty.$$

  \end{proof}
 
\medskip
Next we settle the specific case of Conjecture $1$ where $(k, \ell)=(4,3).$ We begin with the following lemma.
\begin{lemma}\label{kab}
 Let $z_0$ be a zero of $P_n(z)$ and $q = q(z_0)$ be a quotient of two zeros in $t$ of
$D(t,z_0)=A(z_0)t^4 + B(z_0)t^3 + 1.$ Then the set of all such quotients belongs to the curve depicted in Fig. 2 where the equation of the quartic curve $C_5$ (see the left part of Fig. 2) is
$$1 + 2x + 2x^2 + 2x^3 + x^4 - 2y^2 + 2xy^2 + 2x^2y^2 + y^4 = 0,$$
and the other curve $C_4$ (see the right part of Fig. 2) is the segment of the unit circle with real part at least $-1/3.$

 \begin{figure}
    \includegraphics[height=6cm, width=5cm]{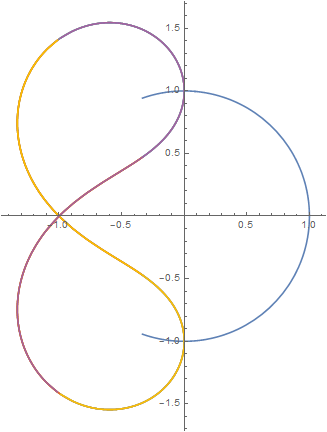}
     \caption{Distribution of the quotients of zeros of the $D(t)=At^4+ Bt^3+1$.}\label{seee1}
 \end{figure}
 \end{lemma}
\begin{proof}
 Let $z_0$ be a zero of $P_n(z)$ such that $A(z_0)B(z_0)\neq 0$ and let $t_1=t_1(z_0),t_2=t_2(z_0), t_3=t_3(z_0)$ and $t_4=t_4(z_0)$ be the zeros of $D(t, z_0)= A(z_0)t^4+ B(z_0)t^3+1$.

By partial fraction decomposition,  we have 
$$\frac{1}{1+B(z_0)t^3+ A(z_0)t^4}=\frac{1}{A(z_0)}\left( \frac{1}{(t-t_1)(t-t_2)(t-t_3)(t-t_4)}\right)=  \frac{1}{A(z_0)}\sum_{i=1}^{4} \frac{1}{t-t_i}  \prod_{i \neq j} \frac{1}{t_i-t_j}.$$
By using geometric progression, we get 
\begin{align*}
\frac{1}{t-t_i}=-\frac{1}{t_i} \sum_{n=0}^{\infty} \frac{t^n}{t_i^n}= -\sum_{n=0}^{\infty} \frac{1}{t_i^{n+1}}t^n. \end{align*}
Hence we obtain
\begin{align}\label{fla}
\frac{1}{D(t,z_0)}=-\frac{1}{A(z_0)}\sum_{n=0}^{\infty}\left(\sum_{i=1}^{4} \frac{1}{t_i^{n+1}}
\prod_{i \neq j} \frac{1}{t_i-t_j}\right)t^n.
\end{align} 
Comparing coefficients of $t^n$ in Equation (\ref{uut}) and Equation (\ref{fla}) we obtain 
\begin{eqnarray*}\label{yu}
P_n(z_0) = -\frac{1}{A(z_0)}\sum_{i=1}^{4} \frac{1}{t_i^{n+1}}
\prod_{i \neq j} \frac{1}{t_i-t_j}.
\end{eqnarray*} 

Let $q_i=t_i/t_1, ~~~~~~ i= 2, 3, 4$.  Now for any $n \geq 0$, the equation $P_n(z_0)=0$ is equivalent to 
 $$
 q_2^{n+1} q_3^{n+3} q_4^{n+2}-q_2^{n+1} q_3^{n+2} q_4^{n+3}-q_2^{n+2} q_3^{n+3} q_4^{n+1}+q_2^{n+2} q_3^{n+1} q_4^{n+3}+q_2^{n+3} q_3^{n+2} q_4^{n+1}-q_2^{n+3} q_3^{n+1} q_4^{n+2}$$ $$ +q_2^{n+1} q_3^{n+2}-q_2^{n+1} q_3^{n+3}+q_2^{n+2} q_3^{n+3}+q_2^{n+3} q_3^{n+1}-q_2^{n+3} q_3^{n+2}-q_2^{n+2} q_3^{n+1}-q_2^{n+1} q_4^{n+2}+q_2^{n+1} q_4^{n+3}$$ $$+q_2^{n+2} q_4^{n+1}-q_2^{n+2} q_4^{n+3}-q_2^{n+3} q_4^{n+1}+q_2^{n+3} q_4^{n+2}-q_3^{n+2} q_4^{n+1}+q_3^{n+3} q_4^{n+1}+q_3^{n+1} q_4^{n+2}-q_3^{n+3} q_4^{n+2}$$ $-q_3^{n+1} q_4^{n+3}+q_3^{n+2} q_4^{n+3}=0.
$

Observe that since $D(t,z_0)= A(z_0)t^4+ B(z_0)t^3+1$, Vieta's formulae give
\begin{eqnarray} \label{oniic}  t_1t_2+t_1t_3+t_1t_4+t_2t_3+ t_2t_4+t_3t_4=0,
\end{eqnarray} and
\begin{eqnarray} \label{on22c}   t_1t_2t_3+t_1t_3t_4+t_1t_2t_4+t_2t_3t_4=0.
\end{eqnarray} 
Divide Equation (\ref{oniic}) by  $t_1^2$ and (\ref{on22c}) by $t_1^3$ and solve the resulting equations simultaneously to get $q_3$ and $q_4$  in terms of $q_2=u$ as follows.
Either 
\begin{eqnarray} \label{on3ii}  q_3= \frac{u(u+1) +i u\sqrt{3 u^2+3+2u}}{2 \left(u^2+u+1\right)}
\quad {\rm and}\quad  q_4= \frac{u(u+1)-iu\sqrt{3 u^2+3+2u}}{2 \left(u^2+u+1\right)},
\end{eqnarray} 
or 
\begin{eqnarray} \label{on322ii}  q_3= \frac{u(u+1) - iu\sqrt{3 u^2+3+2u}}{2 \left(u^2+u+1\right)}
\quad {\rm and}\quad  q_4=  \frac{u(u+1) +iu\sqrt{3u^2+3+2u}}{2 \left(u^2+u+1\right)}.
\end{eqnarray}
Observe that the product and sum of the roots are respectively
\begin{eqnarray} \label{on311} q_3q_4= \frac{u^2}{u^2+u+1}
\quad {\rm and}\quad  q_3+q_4= \frac{-u(u+1)}{u^2+u+1}.
\end{eqnarray} 
From (\ref{on311}), we have that $q_3$ and $q_4$ are two roots of the equation 
\begin{eqnarray} \label{cc}
G(q):= q^2 + \frac{u(u+1)}{u^2+u+1}q+ \frac{u^2}{u^2+u+1}=0.
\end{eqnarray} 
Let $u=e^{i\theta}$ be a point on the unit circle such that $-1/3 \leq \cos \theta \leq 1 $. The quadratic formula    (\ref{cc}) thus gives 
\begin{eqnarray} \label{bw}
q= \frac{1}{2}\left( -1+\frac{e^{-i\theta}}{2\cos \theta +1} \pm \frac{i e^{-i\theta /2}\sqrt{6 \cos \theta +2 }}{2\cos \theta +1}  \right).
\end{eqnarray} 
Splitting the real and imaginary parts of Equation (\ref{bw}), using Mathematica we obtain that $G$ maps the interval 
$-\frac{1}{3} \leq \cos \theta \leq 1$ to the quartic curve 

\begin{eqnarray*} \label{bwi}
1 + 2x + 2x^2 + 2x^3 + x^4 - 2y^2 + 2xy^2 + 2x^2y^2 + y^4 = 0.
\end{eqnarray*}
\end{proof}
\begin{theorem}
Let $ \{P_n(z)\}$ be a polynomial sequence whose generating function is 
\begin{equation*}
 \sum_{n=0}^{\infty}P_n(z)t^n= \frac{1}{1+B(z)t^3+ A(z)t^4},
 \end{equation*}
 where $B(z)$ and $A(z)$ are polynomials in $z$ with complex coefficients. The zeros of $P_n(z)$ for all $n$ which satisfy $A(z)B(z) \neq 0$ lie  on the  real algebraic curve  defined by
 \begin{equation*} \label{im}
\Im\left( \frac{B^4(z)}{A^3(z)}\right)=0 ~~~~~~~~  \mbox{and}~~~~~~~~ 
 0 \leq \Re \left(\frac{B^4(z)}{A^3(z)}\right) < \infty.
  \end{equation*}

\end{theorem}

 \begin{proof}
 Let $z_0$ be a zero of $P_n(z)$ such that $A(z_0)B(z_0)\neq 0$ and let $t_1=t_1(z_0), t_2=t_2(z_0), t_3=t_3(z_0), t_4=t_4(z_0)$ be the zeros of $D(t, z_0)= A(z_0)t^4+ B(z_0)t^3+1$. As above, we consider two cases. 
 
 \textit{Case 1:} Suppose  $D(t,z_0)$ has repeated zeros. The ordinary discriminant of $D(t, z_0)$ is $$Disc_t(D(t,z_0))=4^4A^3(z_0)-3^3B^4(z_0).$$
For repeated zeros,  $Disc_t(D(t,z_0))=0$ which is equivalent to 
 $$\frac{B^4(z_0)}{A^3(z_0)}= \frac{4^4}{3^3}=\frac{256}{27} \in \mathbb{R}.$$ Therefore, the point $z_0$ lies on the curve given by 
 
 \begin{equation*} 
\Im\left( \frac{B^4(z)}{A^3(z)}\right)=0 ~~~~~~~~~~~~~~  \mbox{and}~~~~~~~~~~~~~~ 
 \Re \left(\frac{B^4(z)}{A^3(z)}\right) > 0.
  \end{equation*}
  
\textit{Case 2:} Suppose all the zeros of $D(t, z_0)$ are distinct, i.e., $Disc_t(D(t,z_0)) \neq 0.$
  Let $t_1, t_2,t_3$ and $t_4$ be the  distinct zeros of $D(t,z)$. As above, we  consider the distribution of quotients of zeros $q_i=t_i/t_1, 1\leq i \leq 4$. Then we examine the root distribution of $P_n(z)$ using \textit{q-}discriminants and show that the quotient of zeros $q_i=t_i/t_1, 1\leq i \leq 4$ lie on the curve given in Fig. \ref{seee1}.
 
 By Theorem \ref{ipt1},  the \textit{q-}discriminant of $D(t,z)=1+B(z)t^3 + A(z)t^4$ is given by
 \begin{align*} 
Disc_t(D(t);q)= \frac{B^2 \left(A^3 \left(q^4-1\right)^4-B^4 (q-1) q^3 \left(q^3-1\right)^3\right)}{(q-1)^4}.
 \end{align*}
If $q$ is a quotient of two distinct zeros of $D(t,z_0)=1+B(z_0)t^3 + A(z_0)t^4$ then 
  $Disc_t(D(t);q)=0$ which implies $$ \frac{B^2 \left(A^3 \left(q^4-1\right)^4-B^4 (q-1) q^3 \left(q^3-1\right)^3\right)}{(q-1)^4}=0. $$
Since $q \neq 0,1$, we have

\begin{align} \label{mmayy}
\frac{B^4(z_0)}{A^3(z_0)}= \frac{(q^4-1)^4}{(q^3-1)^3(q^3-q^4)}=:f(q)
\end{align}
 It remains to show that for $q$ on the curve depicted in Fig. \ref{seee1}, we have $ \Im(f(q))=0$. 
To do this, let $u=e^{i\theta}$ be a point on a unit circle with $-\frac{1}{3} \leq \cos \theta \leq 1$. Then $u \in C_4$. To each $u$, there are two possible values of $q$. Moreover,  $u\neq q$ since this would imply that $u=-1\pm i \sqrt{2}$ contradicting that $u$ is a point on the segment of the unit circle with real part at least $-1/3$.
Now, let $u \in C_4$ with the property that $q(u) \in C_5$.  Then $q$ and $u$ are related by the equation

\begin{eqnarray} \label{may}
q^2 + \frac{u(u+1)}{u^2+u+1}q+ \frac{u^2}{u^2+u+1}=0. 
\end{eqnarray}
Multiplying Equation (\ref{may}) by $u- q$, we obtain
\begin{align} \label{mma}
q^3 + q^3 u + q^3 u^2 = u^3 + q u^3 + q^2 u^3.
\end{align}
From (\ref{mmayy}) and the identity (\ref{mma}) we show that $f(q)=f(u)$ as follows.
\begin{align*} f(q)&=\frac{(q^4-1)^4}{(q^3-1)^3(q^3-q^4)}=- \frac{(q+1)^4(q^2+1)^4}{(q^2+q+1)^3q^3}=- \frac{(q+1)^4(q^2+1)^4u^{12}}{q^{12}(u^2+u+1)^3u^3}\\& =- \frac{(1+1/q)^4(1+1/q^2)^4u^{12}}{(u^2+u+1)^3u^3}= \frac{(u^3(1+1/q)(1+1/q^2))^4(u-1)^4}{(u^3-1)^3(u^3-u^4)} \\&= \frac{(u^3(1+1/q+1/q^2+1/q^3))^4(u-1)^4}{(u^3-1)^3(u^3-u^4)}= \frac{(u^3(1+1/u+1/u^2+1/u^3))^4(u-1)^4}{(u^3-1)^3(u^3-u^4)} \\&= \frac{(1+u+u^2+u^3)^4(u-1)^4}{(u^3-1)^3(u^3-u^4)}=\frac{(u^4-1)^4}{(u^3-1)^3(u^3-u^4)}=f(u).
\end{align*}
Since $u \in C_4$, it implies $|u|=1$. Lemma \ref{go2} then gives $\Im(f(u))=\Im(f(q))=0 $. The conclusion $\Im(f)= 0$ thus follows. 
Therefore, the point $z_0$ lies on the curve given by $$\Im \left( \frac{B^3(z)}{A^2(z)}\right)=0. $$
  
Next we prove the inequality constraint.  From the recurrence relation with $\ell= 3, k = 4$ we have
 $$P_n(z)+B(z)P_{n-3}(z)+A(z)P_{n-4}(z)=0$$ whose characteristic equation is $t^4+ B(z)t + A(z)=0.$ The corresponding denominator of the generating function is $D(t,z)=A(z)t^4+ B(z)t^3 + 1.$ Suppressing $z,$ we can write $D(t)=At^4+ Bt^3 + 1.$

Let $t_1, t_2, t_3, t_4$ be the zeros of $D(t)$ labelled so that $0<|t_1| \leq |t_2|\leq |t_3|\leq |t_4|$ and set $q_i:=t_i/t_1$ for  $i=1,2,3,4$. Note these are functions of $z$.  

From the Equations (\ref{on3ii}) and (\ref{on322ii}), it is enough to consider only one pair of solutions $(q_3, q_4)$ since the other pair follows similarly and give the same results. 

If $z_0$ is a zero of $P_n(z)$ and $A(z_0)B(z_0)\neq 0$, then $q_1=1$, $|u|=1$, $|q_3|\geq 1$ and $|q_4|\geq 1.$
So, we search $u$ such that the above conditions are satisfied. Set $u=e^{i\theta}$ where $\theta \in [0, 2\pi]$.
It follows from $|q_3|\geq 1$ that 
\begin{align*}
 \left|\frac{u}{u^2+u+1}\right| \left| u+1+ i\sqrt{3 u^2+3+2u}\right|  \geq 2,
  \end{align*}
  which becomes
\begin{eqnarray}\label{laaila}
\left| 1+ q^{i\theta} + i\sqrt{(3q^{2i\theta}+3+2q^{i\theta})} \right| \geq \left|2 (q^{2i\theta}+ q^{i\theta}+1)\right|.
\end{eqnarray}
Simplifying (\ref{laaila}) gives
\begin{eqnarray}\label{lab1i}
(2 \cos \theta +1)\cos \theta \leq 0.
\end{eqnarray}
Solving  (\ref{lab1i}) gives $\theta \in [\frac{\pi}{2} , \frac{2\pi}{3} ] \cup [\frac{4\pi}{3}, \frac{3\pi}{2}].$ 

Similarly, for $|q_4|\geqslant 1$,  we get $\theta \in [\frac{\pi}{2} , \frac{2\pi}{3} ] \cup [\frac{4\pi}{3}, \frac{3\pi}{2}].$ 
 We thus conclude that both $|q_3|\geq 1$ and $|q_4|\geq 1$ only when $\theta \in [\frac{\pi}{2} , \frac{2\pi}{3} ] \cup [\frac{4\pi}{3}, \frac{3\pi}{2}].$ 
 
Let $\Omega= \{e^{i\theta}:\theta \in [\frac{\pi}{2} , \frac{2\pi}{3} ] \cup [\frac{4\pi}{3}, \frac{3\pi}{2}]\}.$ From
\begin{eqnarray*}\label{tod1i}
f(q)=\frac{B^4(z)}{A^3(z)}=\frac{(q^4-1)^4}{(q^3-1)^3(q^3-q^4)}.
\end{eqnarray*}
If $q \in \Omega$, it follows that 
\begin{eqnarray*}\label{tod1i}
f(q)=\frac{(1+q+q^2+q^3)^4}{q^3(1+q+q^2)^3}
\end{eqnarray*}

We now compute the range of $f$ on $\Omega$. To do this, we set $q= e^{i \theta}$ where $\theta \in [\frac{\pi}{2} , \frac{2\pi}{3}] \cup [\frac{4\pi}{3}, \frac{3\pi}{2} ].$
After this parametrization we obtain
\begin{align*}
f(q)=F(\theta)=\frac{(\cos (4 \theta)-1)^2}{(\cos (3 \theta)-1)(\cos (2 \theta)-\cos \theta)}.
\end{align*}
Its clear that  $F$ is well-defined and continuous on the union $[\frac{\pi}{2},\frac{2\pi}{3} )\cup ( \frac{4\pi}{3}, \frac{3\pi}{2}]$. For $\theta \in \Omega$, we obtain  $0 \leqslant (\cos (4 \theta)-1)^2 \leqslant \frac{9}{4},\quad  -1\leqslant (\cos (3 \theta)-1) \leqslant 0\quad \mbox{and} \quad -1\leqslant (\cos (2 \theta)-\cos \theta)\leqslant 0.$
Hence $F(\theta)\geqslant 0$. Moreover, $F$ attains its minimum value of $0$ at $\theta = \frac{\pi}{2}$ and $\theta = \frac{3\pi}{2}$.
Since $\lim_{\theta \to \frac{2\pi}{3}^-}F(\theta)=+\infty= \lim_{\theta \to \frac{4\pi}{3}^+}f(\theta)$, it follows from continuity and non-negativity of $F$ that the range of $f$ is $\mathbb{R}_{\geqslant 0}$. 
This proves the inequality condition that $$ 0 \leq \Re \left(\frac{B^4(z)}{A^3(z)}\right) < \infty.$$ 
\end{proof}

\section{Examples}\label{BIII}
In this section we present several concrete examples using numerical experiments. In these examples, we consider the sequence of polynomials $\{P_n\}_{n=0}^{\infty}$ generated by the rational function 
\begin{equation*}\label{uuuu}
 \sum_{n=0}^{\infty}P_n(z)t^n= \frac{1}{1+B(z)t^\ell+ A(z)t^k},
 \end{equation*}
where $A(z)$ and $B(z)$ are arbitrary complex polynomials. We plot a portion of the curve  given by 
$$\Im \left(\frac{B^k(z)}{A^\ell(z)}\right)=0.$$ On each graph, we plot the zeros of one of the polynomials (of our choice) in the polynomial sequence described by the given parameters $k, \ell, A(z)$ and  $B(z)$.
\begin{example} For $k=3, \ell=2, A(z)=z+5$ and  $B(z)=-z^2 + 2 z + 5$, we obtain Fig. \ref{a1} and Fig. \ref{a2}. In Fig. \ref{a1}, we plot the zeros of $P_{30}(z)$ while in  Fig. \ref{a2}, we plot the zeros of $P_{70}(z).$
\begin{figure}[!htb]
   \begin{minipage}{0.48\textwidth}
     \centering
     \includegraphics[height=6.5cm, width=6.5cm]{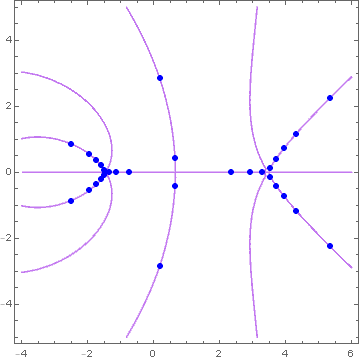}
     \caption{The graph of $\Im(B^3(z)/A^2(z))=0$ and the zeros of $P_{30}(z)$.}\label{a1}
   \end{minipage}\hfill
   \begin{minipage}{0.48\textwidth}
     \centering
     \includegraphics[height=6.5cm, width=6.5cm]{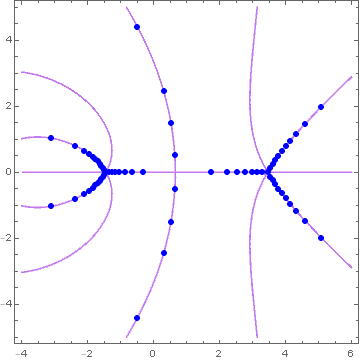}
   \caption{The graph of $\Im(B^3(z)/A^2(z))=0$ and the zeros of $P_{70}(z)$.}\label{a2}
   \end{minipage}
\end{figure}
\end{example}
\begin{example} For $k=3, \ell=2, A(z)=z^3 - z + 6$ and  $B(z)=-z^2 + 7 z - 5$, we obtain Fig. \ref{as1} and 
Fig. \ref{as2}. In Fig.  \ref{as1}, we plot the zeros of $P_{120}(z)$ while in  Fig.  \ref{as2}, we plot the zeros of $P_{200}(z).$
\begin{figure}[!htb]
   \begin{minipage}{0.48\textwidth}
     \centering
     \includegraphics[height=6.5cm, width=6.5cm]{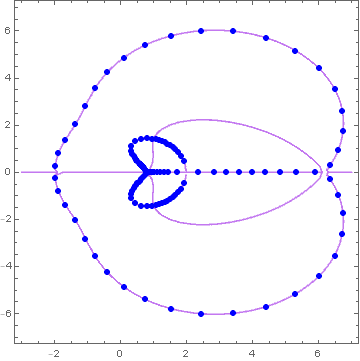}
     \caption{The graph of $\Im(B^3(z)/A^2(z))=0$ and the zeros of $P_{120}(z)$.}\label{as1}
   \end{minipage}\hfill
   \begin{minipage}{0.48\textwidth}
     \centering
     \includegraphics[height=6.5cm, width=6.5cm]{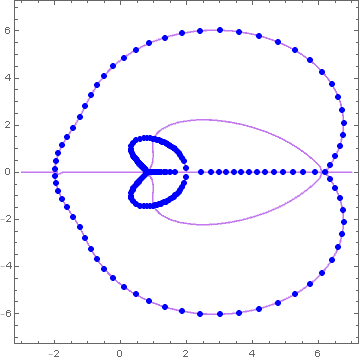}
   \caption{The graph of $\Im(B^3(z)/A^2(z))=0$ and the zeros of $P_{200}(z)$.}\label{as2}
   \end{minipage}
\end{figure}
\end{example}
\begin{example} For $k=4, \ell=3, A(z)=z^2+1$ and  $B(z)=z^3-1$, we obtain  Fig. \ref{figg1} and Fig. \ref{Figg2}. In Fig. \ref{figg1}, we plot the zeros of $P_{40}(z)$ while in  Fig. \ref{Figg2}, we plot the zeros of $P_{70}(z).$
\begin{figure}[!htb]
   \begin{minipage}{0.48\textwidth}
     \centering
     \includegraphics[height=6.5cm, width=6.5cm]{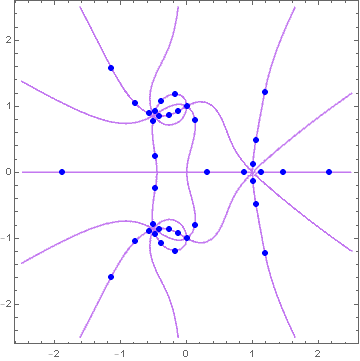}
     \caption{The graph of $\Im(B^4(z)/A^3(z))=0$ and the zeros of $P_{40}(z)$.}\label{figg1}
   \end{minipage}\hfill
   \begin{minipage}{0.48\textwidth}
     \centering
     \includegraphics[height=6.5cm, width=6.5cm]{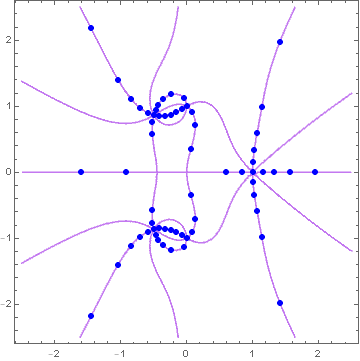}
     \caption{The graph of $\Im(B^4(z)/A^3(z))=0$ and the zeros of $P_{70}(z)$.}\label{Figg2}
   \end{minipage}
\end{figure}
\end{example}
\begin{example} For $k=4, \ell=3, A(z)=7 z^5 - 2 z + i$ and  $B(z)=-z^2 - 2 z + 5$, we obtain  Fig. \ref{fikgg1} and Fig. \ref{Fikgg2}. In Fig. \ref{fikgg1}, we plot the zeros of $P_{50}(z)$ while in  Fig. \ref{Fikgg2}, we plot the zeros of $P_{150}(z).$
\begin{figure}[!htb]
   \begin{minipage}{0.48\textwidth}
     \centering
     \includegraphics[height=6.5cm, width=6.5cm]{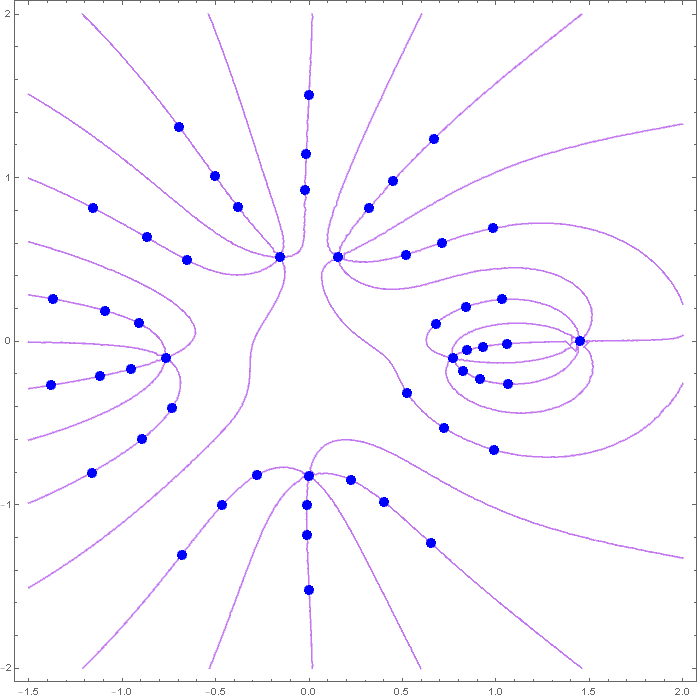}
     \caption{The graph of $\Im(B^4(z)/A^3(z))=0$ and the zeros of $P_{50}(z)$.}\label{fikgg1}
   \end{minipage}\hfill
   \begin{minipage}{0.48\textwidth}
     \centering
     \includegraphics[height=6.5cm, width=6.5cm]{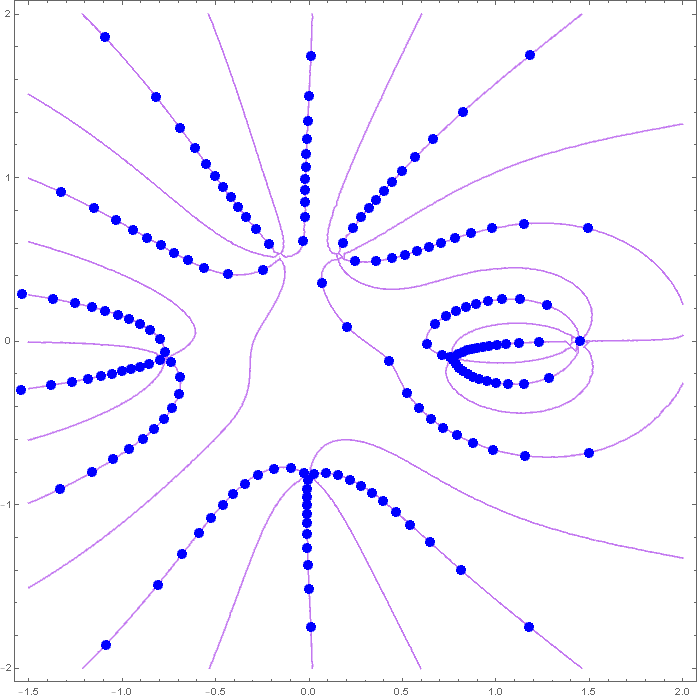}
     \caption{The graph of $\Im(B^4(z)/A^3(z))=0$ and the zeros of $P_{150}(z)$.}\label{Fikgg2}
   \end{minipage}
\end{figure}
\end{example}
We end this paper with the following conjecture which we obtain based on numerical experiments.
\newpage
\vspace*{2cm}
\begin{conjecture} \label{Inn} In the notation of Conjecture \ref{conj:Shapiro} with the condition  $\ell>1,$
we have the following.
\begin{enumerate}
[(a)]
\item If $k$ is even, then $$ 0 \leq \Re \left(\frac{B^k(z)}{A^\ell(z)}\right) < \infty.$$

\item If $k$ is odd, then

$$ 0 \leq - \Re \left(\frac{B^k(z)}{A^\ell(z)}\right) < \infty \quad {\rm for} \quad \ell  \quad {\rm odd}                $$
and 
$$ 0 \leq \Re \left(\frac{B^k(z)}{A^\ell(z)}\right) < \infty \quad {\rm for} \quad \ell  \quad {\rm even.}                $$
\end{enumerate}

\end{conjecture}
\begin{remark} \end{remark}The conjecture excludes the case $\ell =1$  because these were settled \cite{T, TI}. It is important to note that in all those cases, the real part of rational function is bounded unlike in Conjecture $2$ where the real part is unbounded. 

\section{Final Remarks}
 \textbf{Problem:} Give a complete proof of Conjecture 1 as formulated by B. Shapiro. This  still remains unknown to the author and it should be a project for a future work. Additionally to give a complete proof of Conjecture 2 as formulated in this paper. 
\medskip 


{\bf Acknowledgements.}  I am sincerely grateful to my advisors Professor Boris Shapiro who introduced me to the problem and for many fruitful discussions surrounding it, Dr. Alex Samuel Bamunoba, Prof. Rikard B\o gvad, Dr.David Sseevviiri for their continuous discussions and excellent guidance. I acknowledge and appreciate the financial support from Sida Phase-IV bilateral program with Makerere University 2015-2020 under project 316 'Capacity building in mathematics and its applications'.

\end{document}